\documentclass[12pt]{article}
\usepackage[T2A]{fontenc}
\usepackage[cp1251]{inputenc}
\usepackage[english]{babel}
\usepackage{amsmath,amsfonts,amssymb, amsthm}
\usepackage{cite}
\usepackage{xcolor}

\theoremstyle{plain}
\newtheorem*{theorem}{Theorem}
\newtheorem{lemma}{Lemma}
\newtheorem*{corollary}{Corollary}
\theoremstyle{definition}

\newtheorem*{example}{Example}

\begin{document}

\begin{center}
\textsc{\textbf{On grouops with formational subnormal or self-normalizing subgroups}}
\end{center}

\medskip

\begin{center}
I.\,L.~Sokhor
\end{center}

\begin{abstract}
We establish the structure of finite groups with $\mathfrak{F}$-subnormal or self-normalizing primary cyclic subgroups in case $\mathfrak{F}$ is a subgroup-closed saturated superradical formation containing all nilpotent groups.
\end{abstract}


\section{Introduction}

All groups in this paper are finite. We use the standard notation and terminology
of~\cite{Hup,DH,Mon}.

Let $\mathfrak F$ be a formation, and let $G$ be a group.
A subgroup~$H$ is called $\mathfrak F$-subnormal
if either $G=H$ or there is a chain of subgroups
\[H=H_0< \cdot \  H_1< \cdot \  \ldots < \cdot \  H_n=G   \]
such that $H_i/({H_{i-1}})_{H_i}\in \mathfrak F$ for all~$i$,
this is equivalent to $H_i^\mathfrak F\le H_{i-1}$.
Here $A_B=\bigcap _{b\in B}A^b$ is the core of a subgroup $A$ in a group~$B$,
$H_{i-1}<\!\cdot \,H_i$ denotes that $H_{i-1}$ is a maximal subgroup of a group~$H_i$.
A subgroup~$H$ of a group~$G$ is said to be $\mathfrak F$-abnormal in $G$ if
$L/K_L\not\in\mathfrak{F}$ for all subgroups $K$ and $L$
such that $H\le K<\!\cdot\,L\le G$.
It is clear that any proper subgroup of a group can not be both
$\mathfrak{F}$-subnormal and $\mathfrak{F}$-abnormal, i.\,e.
these notions are alternative. Besides,
if  $\mathfrak{X}\subseteq\mathfrak{F}$, then
every $\mathfrak{X}$-subnormal subgroup is $\mathfrak{F}$-subnormal and
every $\mathfrak{F}$-abnormal subgroup is $\mathfrak{X}$-abnormal.

Many authors investigated groups in which all or certain subgroups are
$\mathfrak F$-subnormal or $\mathfrak F$-abnormal,
see references in~\cite{Skiba2016}.

For a subgroup-closed formation $\mathfrak{F}$ containing all nilpotent groups,
every $\mathfrak{F}$-abnormal subgroup is self-normalizing.
Self-normalizingness and $\mathfrak{F}$-subnormality are not alternative notions.
For instance, in a soluble group, every non-normal subgroup of prime index
is both self-normalizing and $\mathfrak{U}$-subnormal.
Here $\mathfrak{U}$ denotes the formation of all supersoluble groups.

\begin{example}
Assume that $\mathfrak{F}=\mathfrak{NA}$ is the formation of all groups
with the nilpotent derived subgroups. The class of groups with
$\mathfrak{F}$-subnormal or
$\mathfrak{F}$-abnormal primary subgroups
was investigated in~\cite{Sem11}.
If we replace $\mathfrak{F}$-abnormality by
self-normalizingness, then the class under study broadens.

By $E_{p^n}$ we denote an elementary abelian group of order $p^n$
for a prime $p$ and a positive integer $n$,
$Z_m$ denotes a cyclic group of order $m$ for a positive integer $m$.

In GAP's SmallGroup library~\cite{GAP}, there is the group
\[G=(S_3\times S_3\times A_4)\rtimes Z_2\quad (\verb"GAP"~SmallGroup~ID~[864,4670]).\]
In $G$, the Sylow $3$-subgroup ${G_3 \simeq E_{3^3}}$ is $\mathfrak{F}$-subnormal,
the Sylow $2$-subgroup $G_2 \simeq E_{2^4}\rtimes Z_2$ is self-normalizing,
non-$\mathfrak{F}$-subnormal and non-$\mathfrak{F}$-abnormal, and
every proper subgroup of $G_2$ is $\mathfrak{F}$-subnormal.
Besides,
\[G^{\mathfrak{F}}=F(G)\simeq E_{3^2}\times E_{2^2}<
G^{\mathfrak{N}}\simeq E_{3^2}\times A_4<
G^{\prime}\simeq (E_{3^2}\times A_4)\rtimes Z_2.
\]
Thus $G$ belongs to the class of groups with
$\mathfrak{F}$-subnormal or self-normalizing
primary subgroups and does not belong to the class of groups in which
primary subgroups are $\mathfrak{F}$-subnormal or
$\mathfrak{F}$-abnormal.
\end{example}

Groups in which certain subgroups are $\mathfrak F$-subnormal
or self-normalizing  were studied in~\cite{Mon2016}--\cite{MonS2018}.
In particular, in~\cite{MonS2018} the structure of group with
$\mathfrak F$-subnormal or self-normalizing Sylow subgroups
was described for the large class of subgroup-closed formations~$\mathfrak{F}$.

We proceed to develop this line of research  and describe groups with $\mathfrak{F}$-subnormal
or self-normalizing primary cyclic subgroups in case $\mathfrak{F}$ is a subgroup-closed
saturated superradical formation containing all nilpotent groups. We prove

\begin{theorem}\label{th_main}
If $\mathfrak{F}$ is a subgroup-closed saturated superradical formation
containing all nilpotent groups, then for a soluble group $G\notin\mathfrak{F}$,
the following statements are equivalent.

$(1)$~Every primary cyclic subgroup of $G$ is self-normalizing or
$\mathfrak{F}$-subnormal.

$(2)$~Every proper subgroup of $G$ is self-normalizing or $\mathfrak{F}$-subnormal.

$(3)$~$G=G^{\prime}\rtimes \langle x\rangle$, where
$\langle x\rangle$ is a Sylow $p$-subgroup for some $p\in\pi(G)$
and a Carter subgroup, $G^{\prime}\rtimes \langle x^p\rangle\in \mathfrak{F}$.
\end{theorem}

A subnormal subgroup-closed formation  $\mathfrak{F}$ is superradical if
a group $G=AB$, where $A$ and $B$ are $\mathfrak {F}$-subnormal
$\mathfrak {F}$-subgroups of $G$, belongs to $\mathfrak {F}$.
It is well known that a formation with the Shemetkov property~\cite[6.4.6]{BalCl}
and a lattice formation~\cite[Lemma~4]{VasKomS93} are superradical.

\section{Preliminaries}

If $A$ is a subgroup of a group~$B$, then we write $A\leq B$;
if $A$ is a normal  subgroup of a group $B$, then we write $A\lhd B$.
By $\pi(G)$ we denote the set of all primes dividing the order of a group $G$.
A semidirect product of a normal subgroup $A$ and a subgroup $B$ is denoted
by $A\rtimes B$. The symbol $\square$ indicates the end of the proof.

The formations of all abelian and nilpotent subgroups are denoted by
$\mathfrak{A}$ and $\mathfrak{N}$, respectively.

Let $\mathfrak F$ be a formation, and $G$ be a group. The subgroup
\[G^\mathfrak{F}=\bigcap\{N\lhd G : G/N\in\mathfrak{F}\}\]
is called the $\mathfrak{F}$-residual of $G$.

If $\mathfrak X$ and $\mathfrak F$ are subgroup-closed formations, then the product
\[
\mathfrak X\mathfrak F=\{~G\in \mathfrak E \mid G^{\mathfrak F}\in \mathfrak{X}\}
\]
is also a subgroup-closed formation according to \cite[p.~337]{DH} and \cite[p.~191]{Mon}.

We need the following properties of  $\mathfrak{F}$-subnormal and
$\mathfrak{F}$-abnormal subgroups.

\begin{lemma}
\label{EF_lem1_1}
Let $\mathfrak{F}$ be a formation, let $H$ and $K$ be subgroups of~$G$,
and let $N\lhd  G$. The following statements hold.

$(1)$~If $K$ is $\mathfrak{F}$-subnormal in $H$ and $H$ is $\mathfrak{F}$-subnormal in~$G$,
      then $K$ is $\mathfrak{F}$-subnormal in $G$~\textup{\cite[6.1.6\,(1)]{BalCl}}.

$(2)$~If $K/N$ is $\mathfrak{F}$-subnormal in $G/N$, then $K$ is $\mathfrak{F}$-subnormal
      in $G$~\textup{\cite[6.1.6\,(2)]{BalCl}}.

$(3)$~If $H$ is $\mathfrak{F}$-subnormal in~$G$, then $HN/N$ is $\mathfrak{F}$-subnormal
      in $G/N$~\textup{\cite[6.1.6\,(3)]{BalCl}}.

$(4)$~If $\mathfrak{F}$ is a subgroup-closed  formation and $G^{\mathfrak{F}}\leq H$,
      then $H$ is $\mathfrak{F}$-subnormal in $G$~\textup{\cite[6.1.7\,(1)]{BalCl}}.

$(5)$~If $\mathfrak{F}$ is a subgroup-closed  formation, $K\leq H$, $H$ is
      $\mathfrak{F}$-subnormal in~$G$ and $H\in\mathfrak{F}$, then $K$ is
      $\mathfrak{F}$-subnormal in $G$.
\end{lemma}

\begin{proof}
(5)~Since $\mathfrak F$ is a subgroup-closed formation and $H\in\mathfrak{F}$,
    we have $K$ is $\mathfrak{F}$-subnormal in $H$ and $K$ is $\mathfrak{F}$-subnormal
    in $G$ in view of~$(1)$.
\end{proof}

\begin{lemma}[{\cite[Lemma 1.4]{Mon2016}}]\label{lem_Fabn}
Let $\mathfrak F$ be a subgroup-closed formation containing groups of order $p$
for all $p\in \mathbb{P}$, and let $A$ be a $\mathfrak F$-abnormal subgroup of $G$.

$(1)$ If $A\le B\le G$, then $A$ is $\mathfrak F$-abnormal in $B$ and $A=N_G(A)$;

$(2)$ If $A\le B\le G$, then $B$ is $\mathfrak F$-abnormal in $G$ and $B=N_G(B)$.
\end{lemma}

A subgroup $H$ of a group $G$ is called an $\mathfrak{X}$-projector of $G$
if $HN/N$ is an $\mathfrak{X}$-maximal subgroup of $G/N$ for every normal
subgroup $N$ of $G$. A Carter subgroup is a nilpotent self-normalizing
subgroup~(\cite[VI.12]{Hup}, \cite[III.4.5]{DH}).
In soluble groups, Carter subgroups are $\mathfrak {N}$-projectors,
they exist and are conjugate. An insoluble group may have no Carter subgroups,
but by E.\,P.~Vdovin theorem~\cite{Vd} Carter subgroups are conjugate whenever they exist.

\begin{lemma}[{\cite[Theorem~15.1]{Shem}}]\label{lem_proj}
Let $\mathfrak{F}$ be a formation. A subgroup $H$ of a soluble group $G$ is
an $\mathfrak{F}$-projector of $G$ if and only if $H\in \mathfrak{F}$ and
$H$ is $\mathfrak{F}$-abnormal in $G$.
\end{lemma}

If $G\notin\mathfrak{F}$, but every proper subgroup of $G$ belongs to $\mathfrak{F}$,
then $G$ is a minimal non-$\mathfrak{F}$-group. A minimal non-$\mathfrak{N}$-group is
also called a Schmidt group, and its properties is well known~\cite{Mon_Sch}.

\begin{lemma}[{\cite[Lemma~3]{Sem96}}]\label{lem_minF}
Let $\mathfrak{F}$ be a subgroup-closed saturated formation.
A soluble minimal non-$\mathfrak{F}$-group $G$ is a group of one of the following types:

$(1)$~$G$ is a group of order $p$ for a prime $p\notin\pi(\mathfrak{F})$;

$(2)$~$G$ is a Schmidt group.
\end{lemma}

\begin{lemma}\label{lem_inF}
Let $\mathfrak{F}$ be a subgroup-closed saturated formation
containing all nilpotent groups. A soluble group $G$ belongs $\mathfrak{F}$
if and only if every primary cyclic subgroup of $G$ is $\mathfrak{F}$-subnormal.
\end{lemma}

\begin{proof}
Assume that $G\in\mathfrak{F}$. Then every proper, and thus every primary cyclic
subgroup of $G$, is $\mathfrak{F}$-subnormal.

Conversely, suppose that there are groups not in  $\mathfrak{F}$,
in which every primary cyclic subgroup is $\mathfrak{F}$-subnormal.
Choose a group $G$ of minimal order among these groups.
Then every proper subgroup of $G$ belongs to $\mathfrak{F}$.
In view of Lemma~\ref{lem_minF}, $G$ is a Schmidt group, and
$G=P\rtimes \langle y\rangle$~\cite[Theorem~1.1]{Mon_Sch}.
By~\cite[Theorem~1.5\,(5.2)]{Mon_Sch}, either $G^\mathfrak{F}\leq\Phi (G)$
or $P\leq G^\mathfrak{F}$. If $G^\mathfrak{F}\leq\Phi (G)$, then $G\in\mathfrak{F}$
since $\mathfrak{F}$ is a saturated formation, a contradiction.
Let $P\leq G^\mathfrak{F}$. By the choice of $G$, $\langle y\rangle$ is
$\mathfrak{F}$-subnormal in $G$, and so in $G$, there is a maximal subgroup $M$
containing $\langle y\rangle$ and $G^{\mathfrak{F}}$, a contradiction.
\end{proof}

\section{The Theorem Proof}

\begin{proof}
Assume that every primary cyclic subgroup of a soluble group $G\notin\mathfrak{F}$ is
self-normalizing or $\mathfrak{F}$-subnormal. Then according to Lemma~\ref{lem_inF},
there is a cyclic $p$-subgroup $\langle x\rangle$  for some $p\in \pi (G)$, which
is not $\mathfrak{F}$-subnormal in $G$. By the choice of  $G$, $\langle x\rangle$ is
self-normalizing, and so $\langle x\rangle$ is a Sylow subgroup and a Carter subgroup
of $G$. Since a Carter subgroup is an $\mathfrak{N}$-projector~\cite[5.27]{Mon},
we get $G=G^{\mathfrak{N}}\langle x\rangle$. In view of~\cite[IV.2.6]{Hup}, in $G$
there is a normal Hall $p^{\prime}$-subgroup $G_{p^{\prime}}$ and
$G=G^{\mathfrak{N}}\langle x\rangle=G_{p^{\prime}}\rtimes\langle x\rangle$.
Hence $G_{p^{\prime}}\leq G^{\mathfrak{N}}$, but
$G/G_{p'}\simeq \langle x\rangle\in\mathfrak{A}\subseteq\mathfrak{N}$ and
$G^\mathfrak{N}\leq G'\leq G_{p'}$.
Thus, $G_{p^{\prime}}=G^{\mathfrak{N}}=G'$ and $G=G'\rtimes \langle x\rangle$.
As Carter subgroups of soluble groups are conjugate~\cite[5.28]{Mon},
we conclude that $G^{\prime}\rtimes \langle x^p\rangle$ has no self-normalizing
primary cyclic subgroup. Therefore $G^{\prime}\rtimes \langle x^p\rangle\in\mathfrak{F}$
by Lemma~\ref{lem_inF}. Thus $(3)$ follows from $(1)$.

Now we prove that $(3)$ implies $(2)$. Assume that a soluble group $G\notin\mathfrak{F}$
is represented in the form $G=G^{\prime}\rtimes \langle x\rangle$, where
$\langle x\rangle$ is a Sylow $p$-subgroup for some $p\in\pi(G)$ and
a Carter subgroup, $G^{\prime}\rtimes \langle x^p\rangle\in \mathfrak{F}$.
Choose a subgroup $H$ of $G$. If $|\langle x\rangle|$ divides  $|H|$,
then $\langle x\rangle ^g\leq H$ for some $g\in G$ and $H$ is self-normalizing.
Suppose that $|\langle x\rangle|$ does not divide $|H|$. Then $A=G^{\prime}H$
is a proper subgroup of $G$, and $A\in\mathfrak{F}$ by the choice of $G$.
We conclude from $\mathfrak{A}\subseteq\mathfrak{N}\subseteq\mathfrak{F}$ that
$G^\mathfrak{F}\leq G'\leq A$, and $A$ is $\mathfrak{F}$-subnormal in $G$
by Lemma~\ref{EF_lem1_1}\,$(4)$. Hence $H$ is $\mathfrak{F}$-subnormal in $G$
in view of Lemma~\ref{EF_lem1_1}\,$(5)$. Thus, $(2)$ follows from $(3)$.

Finally, assume that every proper subgroup of $G$ is self-normalizing or
$\mathfrak{F}$-subnormal. Obviously, every primary cyclic subgroup of $G$
is also self-normalizing or $\mathfrak{F}$-subnormal.
Thus $(2)$ implies $(1)$.
\end{proof}

Note that in view of Lemma~\ref{lem_Fabn}\,(1), if $\mathfrak{F}$ is
a subgroup-closed formation containing all nilpotent subgroups,
then every $\mathfrak{F}$-abnormal subgroup is
self-normalizing. Hence the proved theorem 
extends results of~\cite{EB,SS_PFMT,Sem11,Fatt}. In particular,

\begin{corollary}
If $\mathfrak{F}$ is a subgroup-closed saturated superradical formation
containing all nilpotent groups, then for a soluble group $G\notin\mathfrak{F}$,
the following statements are equivalent.

$(1)$~Every primary cyclic subgroup of $G$ is
$\mathfrak{F}$-subnormal or $\mathfrak{F}$-abnormal.

$(2)$~Every proper  subgroup of $G$ is
$\mathfrak{F}$-subnormal or $\mathfrak{F}$-abnormal.

$(3)$~$G=G^{\prime}\rtimes \langle x\rangle$, where
$\langle x\rangle$ is a Sylow $p$-subgroup for some $p\in\pi(G)$
and an $\mathfrak{F}$-projector of $G$,
$G^{\prime}=G^{\mathfrak{F}}$ and $G^{\prime}\rtimes \langle x^p\rangle\in \mathfrak{F}$.
\end{corollary}

\begin{proof}
Firstly, we prove that $(3)$ follows from $(1)$. Assume that every primary cyclic subgroup
of $G$ is $\mathfrak{F}$-subnormal or $\mathfrak{F}$-abnormal. Then
it follows from Lemma~\ref{lem_Fabn}\,(1) that every primary cyclic subgroup of $G$ is
$\mathfrak{F}$-subnormal or self-normalizing, and we can use the proved theorem. 
So, $G=G^{\prime}\rtimes \langle x\rangle$, where $\langle x\rangle$ is a  Sylow
$p$-subgroup for some $p\in\pi(G)$ and a Carter subgroup,
$G^{\prime}\rtimes \langle x^p\rangle\in \mathfrak{F}$.
To prove that $\langle x\rangle$ is $\mathfrak{F}$-abnormal in $G$,
we suppose that is not true. Then $\langle x\rangle$ is
$\mathfrak{F}$-subnormal in $G$ by the choice of $G$.
Hence every primary cyclic subgroup of $G$ is $\mathfrak{F}$-subnormal and $G\in\mathfrak{F}$ by Lemma~\ref{lem_inF}, a contradiction.
Thus $\langle x\rangle$ is $\mathfrak{F}$-abnormal in $G$
and an $\mathfrak{F}$-projector of $G$ in view of Lemma~\ref{lem_proj}.
Therefore $G=G^\mathfrak{F}\langle x\rangle$ and $G'=G^\mathfrak{F}$.

Now assume that $(3)$ is true. According to Lemma~\ref{lem_proj}, we deduce that
$\langle x\rangle$ is $\mathfrak{F}$-abnormal in $G$.
By the proved theorem, 
every subgroup of $G$ is self-normalizing or $\mathfrak{F}$-subnormal.
Let $H$ be a self-normalizing and non-$\mathfrak{F}$-subnormal subgroup of $G$.
If $A=G'H$ is a proper subgroup of $G$, then $A\in\mathfrak{F}$ and $H$ is $\mathfrak{F}$-subnormal in $G$ by Lemma~\ref{EF_lem1_1}\,(5), a contradiction.
Hence $G=G^{\prime}\rtimes \langle x\rangle=G'H$ and $\langle x\rangle\leq H$.
In view of Lemma~\ref{lem_Fabn}\,(2),  we obtain $H$ is $\mathfrak{F}$-abnormal in $G$.
Thus $(3)$ implies $(2)$.

Finally, assume that every proper subgroup of $G$ is $\mathfrak{F}$-subnormal
or $\mathfrak{F}$-abnormal. Then every primary cyclic subgroup of $G$
is also $\mathfrak{F}$-subnormal or $\mathfrak{F}$-abnormal.
Thus $(1)$ follows from $(2)$.

\end{proof}

\textbf{Sokhor Irina Leonidovna}

Brest State University named after A.\,S. Pushkin,

Kosmonavtov Boulevard, 21, 224016 Brest, Republic of Belarus

+375\,162\,21\,71\,14

\texttt{irina.sokhor@gmail.com}

\end{document}